\documentclass[oneside,10pt]{article}          
\usepackage[b5paper]{geometry}	    
\usepackage{amsfonts,amsmath,latexsym,amssymb} 
\usepackage{theorem}                
\usepackage{mathrsfs,upref}         
\usepackage{mathptmx}		    
\usepackage{MIA_arXiv}	            

\newtheorem{theorem}{Theorem}
\newtheorem{lemma}{Lemma}
\newtheorem{corollary}{Corollary}

\theoremstyle{definition}

\begin{document}

\title[Direct and inverse approximation theorems]{Direct and inverse approximation theorems of functions
in the Musielak-Orlicz type spaces}


\author{Fahreddin Abdullayev, Stanislav Chaichenko and Andrii Shidlich}

\address{Fahreddin Abdullayev, \\
                   Faculty of Sciences,\\ 
                   Kyrgyz-Turkish Manas University,\\
                   56, Chyngyz Aitmatov avenue, Bishkek, Kyrgyz republic, 720044;\\
                   Faculty of Science and Letters,\\
                   Mersin University,\\
                   \c{C}iftlikk\"{o}y Kamp\"{u}s\"{u}, Yeni\c{s}ehir, Mersin, Turkey, 33342,\\
\email{fahreddin.abdullayev@manas.edu.kg, fahreddinabdullayev@gmail.com}}

\address{Stanislav Chaichenko,\\
                   Donbas State Pedagogical University,\\
                   19, G.~Batyuka st., Slaviansk, Donetsk region, Ukraine, 84116,\\
\email{s.chaichenko@gmail.com}}

\address{Andrii Shidlich,\\
                   Department of Theory of Functions,\\
                   Institute of Mathematics of NAS of Ukraine,\\
                   3, Tereshchenkivska str., Kyiv, Ukraine, 01601, \\
\email{shidlich@gmail.com}}

\CorrespondingAuthor{Fahreddin Abdullayev}


\date{07.04.2020}                               



\thanks{This work was supported in part
 by the Kyrgyz-Turkish Manas University (Bishkek / Kyrgyz Republic), project No.~KTM\"{U}-BAP-2019.FBE.02,
 the Ministry of Education and Science of Ukraine in the framework of the fundamental research No.~0118U003390 
 and by the European Union's Horizon 2020 research and innovation program under the Marie Sk{\l}odowska-Curie grant agreement No 873071}

\begin{abstract}
     In Musilak-Orlicz type spaces ${\mathcal S}_{\bf M}$, direct and inverse approximation theorems are obtained in
 terms of the best approximations of functions and generalized moduli of smoothness. The  question of the exact constants in Jackson-type inequalities is studied.
\end{abstract}

\maketitle



\section{Introduction}
In  Musilak-Orlicz type  spaces ${\mathcal S}_{\bf M}$, we prove direct and inverse approximation theorems 
in terms of the best approximations of functions and generalized moduli of smoothness. Such theorems establish a connection between the smoothness properties of functions and the behavior of the error of their approximation by various methods.
 In particular, direct theorems show that good smoothness properties of a function
 (the existence of derivatives of a given order, the specific behavior of
 the modulus of smoothness, etc.) imply a good estimate of the error of its  approximation.
  In the case of best approximation by polynomials, these results
are also known as Jackson-type theorems or Jackson-type inequalities \cite{Jackson_1911}.
Inverse theorems characterize smoothness  properties of functions depending on the rapidity with which
the errors of best, or any other, approximations tend to zero. The problem of obtaining inverse theorems in the
approximation of functions was first stated, and in some cases solved, by Bernstein \cite{Bernstein_1912}.
In ideal cases, the direct and inverse theorems complement each other, and this allows us to fully characterize
a functional class having certain smoothness properties, using, for example, sequences of best approximations.
The results concerning direct and inverse connection between the smoothness properties of functions and the errors of their approximations in classical functional spaces (such as Lebesgue and Hilbert spaces, the spaces of continues functions, etc) are   described quite fully in the monographs  \cite{A_Timan_M1960}, \cite{Butzer_Nessel_M1971}, \cite{DeVore_Lorentz_M1993}, \cite{Dzyadyk_Shevchuk_M2008}, \cite{M_Timan_M2009} and others.

In 2001, Stepanets \cite{Stepanets_2001} considered the spaces  ${\mathcal S}^p={\mathcal S}^p({\mathbb T})$ of $2\pi$-periodic  Lebe\-sgue summable functions $f$ ($f\in L$) with the finite norm
 \begin{equation}\label{norm_Sp}
    \|f\|_{_{\scriptstyle {p}}}:=
    \|f\|_{_{\scriptstyle {\mathcal S}^p}}=\|\{\widehat f({k})\}_{{k}\in\mathbb  Z}
    \|_{l_p({\mathbb Z})}:=\Big(\sum_{{ k}\in\mathbb  Z}|\widehat f({k})|^p\Big)^{1/p},
\end{equation}
where
$\widehat{f}(k):={[f]}\widehat{\ \ }(k)={(2\pi)^{-1}}\int_0^{2\pi}f(x) \mathrm{e}^{- \mathrm{i}kx}\mathrm{d}x$,
$k\in\mathbb Z$, are the Fourier coefficients of the function $f$, and  investigated some approximation characteristics of these spaces. Stepanets and Serdyuk \cite{Stepanets_Serdyuk_2002} introduced the notion of $k$th modulus of smoothness in ${\mathcal S}^p$ and proved direct and inverse theorems on approximation in terms of these moduli of
smoothness  and the best approximations of  functions.  Also this topic was
investigated actively in \cite{Sterlin_1972}, \cite{Vakarchuk_2004}, \cite{Vakarchuk 2005}, \cite[Ch.~11]{Stepanets_M2005}, \cite[Ch.~3]{M_Timan_M2009}, etc.

In  \cite{Chaichenko_Shidlich_Abdullayev_2019} and \cite{Abdullayev_Chaichenko_Imash kyzy_Shidlich_2020}, some results for the spaces ${\mathcal S}^p$ were extended  to the Orlicz type  spaces  ${\mathcal S}_{M}$ and ${\mathcal S}_{_{\scriptstyle  \mathbf p,\,\mu}}$. In particular, in \cite{Chaichenko_Shidlich_Abdullayev_2019} and \cite{Abdullayev_Chaichenko_Imash kyzy_Shidlich_2020}, direct and inverse approximation theorems were proved   in terms of best approximations of functions and
moduli of smoothness of fractional order and a connection   was established between $K$-functional and such moduli of smoothness.
In other Banach spaces, in particular, in   Banach spaces of Orlicz type, topics related to  direct and inverse approximation theorems, were  investigated  in  \cite{Guven-Israf-JMI-2010}, \cite{Akgun-Kokilash-GergMJ-2011}, \cite{Jafarov-MIA-2012}, \cite{Jafarov-JMI-2013}, \cite{Sharapudinov-AzGM-2014}, \cite{Akgun-Yildirir-MIA-2015} and others.

Here, 
 we continue such studies
and consider the Musilak-Orlicz type spaces ${\mathcal S}_{\bf M}$, which are natural generalizations of the
spaces ${\mathcal S}_{M}$ and ${\mathcal S}_{_{\scriptstyle  \mathbf p,\,\mu}}$. In these spaces, we give
direct and inverse approximation theorems in terms of best approximations of functions and
generalized moduli of smoothness. Particular attention is paid to the study of the accuracy of constants in Jackson-type inequalities.


\section{Preliminaries }

Let ${\bf M}=\{M_k(u)\}_{k\in {\mathbb Z}}$, $u\ge 0$, be a sequence of Orlicz functions. In other words, for every $k\in {\mathbb Z}$, the
function $M_k(u)$ is a nondecreasing convex function for which $M_k(0)=0$ and $M_k(u)\to \infty$ as $u\to \infty$.
The modular space (or  Musilak-Orlicz  space) ${\mathcal S}_{\bf M}$  is the space of all functions $f\in L$ such that
the following quantity (which is also called the Luxemburg norm of $f$) is finite:
\begin{equation}\label{def_Lux_norm}
    \|{f}\|_{_{\scriptstyle  {\bf M}}}:=
    \|\{\widehat{f}(k)\}_{k\in {\mathbb Z}}\|_{_{\scriptstyle l_{\bf M}({\mathbb Z})}}:=
    \inf\bigg\{a>0:\  \sum\limits_{k\in\mathbb Z}  M_k(|{\widehat{f}(k)}|/{a})\le 1\bigg\}.
\end{equation}
By definition, we say that the functions $f\in L$ and $g\in L$ are assumed to be equivalent in the space $
{\mathcal S}_{\bf M}$, when $\|f-g\|_{_{\scriptstyle  {\bf M}}}\!=\!0.$

 The spaces ${\mathcal S}_{\bf M}$ defined in this way are Banach spaces.
  Sequence spaces of this type have been studied by mathematicians since the 1940s (see, for example, the monographs
  \cite
  {Lindenstrauss-1977}, \cite{Musielak-1983}). If all functions $M_k$ are identical (namely, $M_k(u)\equiv M(u)$,
 $k\in {\mathbb Z}$), the spaces ${\mathcal S}_{\bf M}$ coincide with the ordinary Orlicz type spaces ${\mathcal S}_{M}$
 \cite{Chaichenko_Shidlich_Abdullayev_2019}.
If $M_k(u)=\mu_k u^{p_k}$,  $p_k\ge 1$, $\mu_k\ge 0 $, then ${\mathcal S}_{\bf M}$ coincide with the weighted spaces
${\mathcal S}_{_{\scriptstyle  \mathbf p,\,\mu}}$ with variable exponents \cite{Abdullayev_Chaichenko_Imash kyzy_Shidlich_2020}.
If all $M_k(u)=u^p$, $p\ge 1$, then the spaces ${\mathcal S}_{\bf M}$ are the above-defined spaces ${\mathcal S}^p$.


In addition  to the Luxembourg norm (\ref{def_Lux_norm}), in the space $
{\mathcal S}_{\bf M}$, consider  the
Orlicz norm that is defined as follows. Let
 ${\bf \tilde{M}}=\{\tilde{M}_k(v)\}_{k\in {\mathbb Z}}$ be the sequence of  functions defined by the relations
 \[
    \tilde{M}_k(v):=\sup\{uv-M_k(u): ~u\ge 0\}, \quad k \in \mathbb{Z}.
 \]
Consider the set $\Lambda{=}\Lambda({\bf \tilde{M}})$ of 
sequences of positive numbers $\lambda=\{\lambda_k\}_{k\in \mathbb{Z}}$
such that  $\sum_{k\in \mathbb{Z}}\tilde{M}_k(\lambda_k){\le} 1$. For any function  $f\in {\mathcal S}_{\bf M}$, define its Orlicz norm by the equality
\begin{equation} \label{def-Orlicz-norm}
    \|f\|^\ast_{_{\scriptstyle  {\bf M}}}:=\|\{\widehat{f}(k)\}_{k\in {\mathbb Z}}\|_{_{\scriptstyle l_{\bf M}^*({\mathbb Z})}}:= \sup \Big\{ \sum\limits_{k \in \mathbb{Z}}
    \lambda_k|\widehat{f}(k) |: \quad  \lambda\in \Lambda\Big\}.
\end{equation}

The  following  auxiliary Lemma \ref{Lemma_3} establishes the equivalence of the Luxembourg norm (\ref{def_Lux_norm}) and the
Orlicz norm (\ref{def-Orlicz-norm}).

\begin{lemma}
      \label{Lemma_3}
      For any function $f \in {\mathcal S}_{\bf M}$, the following relation holds:
      \begin{equation} \label{estim-for-norms}
      \| f\|_{_{\scriptstyle  {\bf M}}} \le \| f\|^\ast_{_{\scriptstyle  {\bf M}}}\le 2 \,\| f\|_{_{\scriptstyle  {\bf M}}}.
      \end{equation}
\end{lemma}

Relation (\ref{estim-for-norms}) follows from the similarly relation for corresponding norms in the modular Orlicz sequence  spaces
(see, for example \cite[Ch. 4
]{Lindenstrauss-1977}).

Further,   denote by  $\|\cdot\|$  one of the norms $ \| \cdot\|_{_{\scriptstyle  {\bf M}}}$ or
$\| \cdot\|^\ast_{_{\scriptstyle  {\bf M}}}$.


Let   ${\mathcal T}_{n}$, $n=0,1,\ldots$, be the set of 
trigonometric polynomials
${t}_{n}(x) = \sum_{|k|\le n}  c_{k}\mathrm{e}^{\mathrm{i}kx}$ of the order $n$, where $c_{ k}$
are arbitrary complex numbers. For any $f\in {\mathcal S}_{\bf M}$, 
denote by $E_n (f)_{_{\scriptstyle  {\bf M}}}$ and $E_n (f)_{_{\scriptstyle  {\bf M}}}^*$  the best approximations of $f$ by trigonometric polynomials ${t}_{n-1}\in {\mathcal T}_{n-1}$ in the space ${\mathcal S}_{\bf M}$ with respect to the norms $ \| \cdot\|_{_{\scriptstyle  {\bf M}}}$ and $\| \cdot\|^\ast_{_{\scriptstyle  {\bf M}}}$ respectively, i.e.,
\begin{equation}\label{S_M.3}
    E_n (f)_{_{\scriptstyle  {\bf M}}}:=
    \inf\limits_{{t}_{n-1}\in {\mathcal T}_{n-1} }\|f-{t}_{n-1}\|_{_{\scriptstyle  {\bf M}}}
    \quad {\rm and} \quad    E_n (f)_{_{\scriptstyle  {\bf M}}}^*:=
    \inf\limits_{{t}_{n-1}\in {\mathcal T}_{n-1} }\|f-{t}_{n-1}\|_{_{\scriptstyle  {\bf M}}}^*.
\end{equation}


 The  following auxiliary  Lemma \ref{Lemma_Best_app}  characterizes the polynomial of the best approximation  in ${\mathcal S}_{\bf M}$.

\begin{lemma}
           \label{Lemma_Best_app}
           Assume that $f \in {\mathcal S}_{\bf M}$. Then
           \begin{equation} \label{Best_app_all}
           E_n (f):= \inf\limits_{{t}_{n-1}\in {\mathcal T}_{n-1} }\|f-{t}_{n-1}\|=\|f-{S}_{n-1}({f})\|,
           \end{equation}
           where $S_{n-1}(f)=S_{n-1}(f,\cdot)= \sum _{|k|\le n-1}\widehat{f}(k) {\mathrm{e}^{\mathrm{i}k\cdot}}$
           is the Fourier sum of the function $f$.
\end{lemma}


\begin{proof}
Indeed, for any polynomial
${t}_{n-1}=\sum_{|k|\le n-1}  c_{k}\mathrm{e}^{\mathrm{i}k\cdot}\in {\mathcal T}_{n-1}$,   the quantities  $|(f-{t}_{n-1})\widehat{\ \ }(k)|=|\widehat{f}(k)-c_k|$ when $|k|\le n-1$ and $|(f-{t}_{n-1})\widehat{\ \ }(k)|=|\widehat{f}(k)|$ when  $|k|\ge n$. Therefore, in view of (\ref{def_Lux_norm}) and (\ref{def-Orlicz-norm}), the infimum in (\ref{Best_app_all}) is reached in the case when all
$c_k=\widehat{f}(k)$, i.e., when ${t}_{n-1}=S_{n-1}(f)$. $\hfill\Box$

\end{proof}

Let $\omega_\alpha(f,\delta)$ be the modulus of smoothness of a function  $f \in {\mathcal S}_{\bf M}$ of order
$\alpha>0$, i.e.,
\begin{equation}\label{usual_modulus}
    \omega_\alpha(f,\delta):=
    \sup\limits_{|h|\le \delta}\|\Delta_h^\alpha f\|=
     \sup\limits_{|h|\le \delta} \Big\|\sum\limits_{j=0}^\infty (-1)^j {\alpha \choose j} f(\cdot-jh)
     \Big\|,
\end{equation}
where ${\alpha \choose j}=\frac {\alpha(\alpha-1)\cdot\ldots\cdot(\alpha-j+1)}{j!}$ for $j \in \mathbb{N}$ and ${\alpha \choose j}=1$ for $j=0$. By the definition, for any ${k}\in {\mathbb Z}$, we have
\begin{equation}\label{modulus_difference_Fourier_Coeff}
|{[\Delta_h^\alpha f]}\widehat {\ \ }(k)|=|1-\mathrm{e}^{-\mathrm{i}kh}|^\alpha |\widehat{f}(k)|
 =2^\alpha
\Big|\sin \frac{kh}2\Big|^\alpha |\widehat{f}(k)|.
\end{equation}
Now consider the set $\Phi$ of all continuous bounded nonnegative 
pair functions $\varphi$ such that
$\varphi(0)=0$ and the Lebesgue measure of the set $\{t\in {\mathbb R}:\,\varphi(t)=0\}$ is equal to zero. For a fixed function $\varphi\in \Phi$, $h\in {\mathbb R}$ and for any $f \in {\mathcal S}_{\bf M}$, we denote by
$\{{[\Delta_h^\varphi f]}\widehat {\ \ }(k)\}_{k\in {\mathbb Z}}$ the sequence of numbers such that  for any $k\in {\mathbb Z}$,
  \begin{equation}\label{modulus_generalize difference_Fourier_Coeff}
  {[\Delta_h^\varphi f]}\widehat {\ \ }(k)  =\varphi(kh)  \widehat{f}(k).
\end{equation}
If there exists a function $\Delta_h^\varphi f\in L$ whose Fourier coefficients coincide with the numbers
${[\Delta_h^\varphi f]}\widehat {\ \ }(k)$, $k\in {\mathbb Z}$, then, as above,
the expressions $ \| \Delta_h^\varphi f\|_{_{\scriptstyle  {\bf M}}}$
 and  $\|\Delta_h^\varphi f\|^\ast_{_{\scriptstyle  {\bf M}}}$ denote   Luxemburg and Orlicz norms of the function
 $\Delta_h^\varphi f$.
 If such a function does not exist, then we also keep
 the  notation   $ \| \Delta_h^\varphi f\|_{_{\scriptstyle  {\bf M}}}$ and
 $\|\Delta_h^\varphi f\|^\ast_{_{\scriptstyle  {\bf M}}}$.
 But in this case, by these notations
we mean the corresponding norm $\|\cdot\|_{_{\scriptstyle l_{\bf M}({\mathbb Z})}}$
or $\|\cdot\|_{_{\scriptstyle l_{\bf M}^*({\mathbb Z})}}$ of the sequence $\{{[\Delta_h^\varphi f]}\widehat {\ \ }(k)\}_{k\in {\mathbb Z}}$.  Also we denote by $ \| \Delta_h^\varphi f\|$ any of the expressions $ \| \Delta_h^\varphi f\|_{_{\scriptstyle  {\bf M}}}$ and  $\|\Delta_h^\varphi f\|^\ast_{_{\scriptstyle  {\bf M}}}$

Similarly  to \cite{Shapiro_1968}, \cite{Boman_Shapiro_1971}, \cite{Boman_1980}, \cite{Kozko_Rozhdestvenskii_2004},
define  the generalized  modulus of smoothness  $\omega_\varphi$ of a function $f \in {\mathcal S}_{\bf M}$  by the equality:
\begin{equation}\label{general_modulus}
    \omega_\varphi(f,\delta)=\sup\limits_{|h|\le \delta} \|\Delta_h^\varphi f\|.
 \end{equation}
In particular, we  set
\[
    \omega_\varphi(f,\delta)_{_{\scriptstyle  {\bf M}}}:=
    \sup\limits_{|h|\le \delta}\|\Delta_h^\varphi f\|_{_{\scriptstyle  {\bf M}}}\quad {\rm and} \quad
     \omega_\varphi(f,\delta)_{_{\scriptstyle  {\bf M}}}^*:=
    \sup\limits_{|h|\le \delta}\|\Delta_h^\varphi  f\|_{_{\scriptstyle  {\bf M}}}^*.
\]
It follows from  (\ref{modulus_difference_Fourier_Coeff}) that $\omega_\alpha(f,\delta)=\omega_\varphi(f,\delta)$  when $\varphi(t)=2^\alpha
 |\sin (t/2) |^\alpha$.



\section{Direct approximation theorems}


In this section, we  prove direct approximation theorems    in the space
${\mathcal S}_{\bf M}$ in terms of the best approximations and generalized moduli of smoothness,
and also establish Jackson type inequalities with the constants that are the  best possible in some important cases.


Let $V(\tau)$, $\tau>0$,  be a set of bounded nondecreasing functions  $v$ that differ from a constant on  $[0, \tau]$.


\begin{theorem}
      \label{Theorem_2.1}
      Assume that $ f\in {\mathcal S}_{\bf M}$. Then for any $\tau >0$, $n\in {\mathbb N}$ and $\varphi\in \Phi$
      , the following inequality holds:
      \begin{equation}\label{En<omega}
      E_n (f)_{_{\scriptstyle  {\bf M}}}^*\le C_{n,\varphi}(\tau )\, \omega_\varphi\Big(f, \frac {\tau }n\Big)
      _{_{\scriptstyle  {\bf M}}}^*,
      \end{equation}
      where
      \begin{equation}\label{C_n,varphi,p}
       C_{n,\varphi}(\tau ):= \inf\limits _{v  \in  V(\tau )} \frac {v   (\tau ) - v   (0)}{
        I_{n,\varphi}(\tau ,v   )} ,
      \end{equation}
      and
      \begin{equation}\label{I_n,varphi}
      I_{n,\varphi}(\tau ,v   ):=
      \inf\limits _{k \in {\mathbb N}:k \ge n} \int\limits _0^{\tau }\varphi\Big(\frac {k u}n\Big) dv   (u).
      \end{equation}
      In this case, there exists a function $v^*  \in  V(\tau )$ that realizes the greatest lower bound in (\ref{I_n,varphi}).
\end{theorem}

 \begin{proof}
  Let $ f\in {\mathcal S}_{\bf M}$, $n\in {\mathbb N}$ and $h\in {\mathbb R}$. According to (\ref{Best_app_all}) and (\ref{def-Orlicz-norm}), we have
 \begin{equation}\label{E_n^*}
      E_n (f)_{_{\scriptstyle  {\bf M}}}^*= \|f-{S}_{n-1}({f})\|_{_{\scriptstyle  {\bf M}}}^*=
    \sup \Big\{ \sum\limits_{|k| \ge n}
    \lambda_k |\widehat{f}(k) |: \   \lambda\in \Lambda\Big\},
 \end{equation}
and by the definition of supremum, for  arbitrary $\varepsilon>0$  there exists a sequence $\tilde{\lambda} \in \Lambda$,
$\tilde{\lambda}=\tilde{\lambda}(\varepsilon)$, such that the following relations holds:
\[
    \sum\limits_{|k| \ge n} \tilde{\lambda}_k |\widehat{f}(k)|+\varepsilon  \ge \sup \Big\{ \sum\limits_{|k| \ge n}
    \lambda_k |\widehat{f}(k) |: \   \lambda\in \Lambda\Big\}.
\]
In view of (\ref{def-Orlicz-norm}) and (\ref{modulus_generalize difference_Fourier_Coeff}), we have
\[
    \|\Delta _h^\varphi f\|_{_{\scriptstyle  {\bf M}}}^* \ge
    \sup \Big\{ \sum\limits_{|k|\ge n}
    \lambda_k\varphi(kh)|\widehat{f}(k) |: \   \lambda\in \Lambda\Big\}
    \ge \sum\limits_{|k|\ge n}
    \tilde{\lambda}_k \varphi(kh)|\widehat{f}(k)|=
\]
\[
   = \frac{I_{n,\varphi}(\tau ,v   )}{v (\tau)-v (0)}\sum\limits_{|k|\ge n}
    \tilde{\lambda}_k |\widehat{f}(k) |+   \sum\limits_{|k|\ge n}
    \tilde{\lambda}_k |\widehat{f}(k) | \Big(\varphi(kh)-\frac{I_{n,\varphi}(\tau ,v   )}{v (\tau)-v (0)}\Big).
\]
For  any $ u\in [0,\tau]$, we get
 \begin{equation}\label{diff_est}
    \|\Delta _{\frac un}^\varphi f\|_{_{\scriptstyle  {\bf M}}}^* \ge
     \frac{I_{n,\varphi}(\tau ,v   )}{v (\tau)-v (0)}\sum\limits_{|k|\ge n}
    \tilde{\lambda}_k |\widehat{f}(k) |
         +   \sum\limits_{|k|\ge n} \tilde{\lambda}_k|\widehat{f}(k) | \bigg(\varphi\Big(\frac {k u}n\Big)-\frac{I_{n,\varphi}(\tau ,v   )}{v (\tau)-v (0)}\bigg).
 \end{equation}
The both sides of inequality (\ref{diff_est}) are nonnegative  and, in view of the
boundedness of the function $\varphi$, the series on its right-hand side is majorized on the entire
real axis by the absolutely convergent series
  $C(\varphi)\sum_{|k|\ge n} \tilde{\lambda}_k|\widehat{f}(k) |$, where $C(\varphi):=\max_{u\in {\mathbb R}} \varphi(u)$. Then
 integrating this inequality with respect to $dv  (u)$  from  $0$  to $\tau,$ we get
 \[
    \int\limits  _0^{\tau }\|\Delta _{\frac un}^\varphi f\|_{_{\scriptstyle  {\bf M}}}^*dv  \ge
     I_{n,\varphi}(\tau ,v   )\sum\limits_{|k|\ge n} \tilde{\lambda}_k |\widehat{f}(k) |
    +   \sum\limits_{|k|\ge n}
    \tilde{\lambda}_k|\widehat{f}(k) | \bigg(\int\limits  _0^{\tau } \varphi\Big(\frac {k u}n\Big)dv - I_{n,\varphi}(\tau ,v   ) \bigg).
 \]
By virtue of the definition of $I_{n,\varphi}(\tau ,v   )$, we see that the second term on the right-hand side of
the last relation
is nonnegative. Therefore, for any function $v  \in  V(\tau )$,  we have
\[
    \int\limits  _0^{\tau }\|\Delta _{\frac un}^\varphi f\|_{_{\scriptstyle  {\bf M}}}^*dv  \ge
    I_{n,\varphi}(\tau ,v   )  \sum\limits_{|k|\ge n}
    \tilde{\lambda}_k |\widehat{f}(k) |\ge
    I_{n,\varphi}(\tau ,v   )\bigg( \sup \Big\{ \sum\limits_{|k| \ge n}
    \lambda_k |\widehat{f}(k) |: \   \lambda\in \Lambda\Big\} -\varepsilon \bigg),
\]
wherefrom due to an arbitrariness of choice of the number  $\varepsilon$, we conclude that the inequality
\[
    \int\limits  _0^{\tau }\|\Delta _{\frac un}^\varphi f\|_{_{\scriptstyle  {\bf M}}}^*dv  \ge
     I_{n,\varphi}(\tau ,v   ) E_n (f)_{_{\scriptstyle  {\bf M}}}^* 
\]
is true.  Hence,
 \[
E_n (f)_{_{\scriptstyle  {\bf M}}}^* \le \frac 1{I_{n,\varphi}(\tau ,v   )}\int\limits  _0^{\tau }\|\Delta _{\frac un}^\varphi f\|_{_{\scriptstyle  {\bf M}}}^*dv
\le \frac 1{I_{n,\varphi}(\tau ,v   )}\int\limits  _0^{\tau }\omega_\varphi \Big(f,\frac un\Big)_{_{\scriptstyle  {\bf M}}}^* dv,
\]
whence  taking into account nondecreasing of the function $\omega_\varphi$,
 we immediately obtain relation (\ref{En<omega}).
The existence of the function $v^*  \in  V(\tau )$  realizing the greatest lower bound in (\ref{I_n,varphi})
 will be given below in the proof of   Theorem \ref{S^p_Theorem}. $\hfill\Box$

 \end{proof}



\begin{corollary}
      \label{Corollary 01}
      Assume that $ f\in {\mathcal S}_{\bf M}$. Then for any $\tau >0$, $n\in {\mathbb N}$ and
      $\varphi\in \Phi$
      , the following inequality holds:
      \begin{equation}\label{En<omega_M}
      E_n (f)_{_{\scriptstyle  {\bf M}}}\le 2C_{n,\varphi}(\tau )\, \omega_\varphi\Big(f, \frac {\tau }n\Big)
      _{_{\scriptstyle  {\bf M}}} ,
      \end{equation}
      where the quantity $C_{n,\varphi}(\tau )$ is defined by (\ref{C_n,varphi,p}).
\end{corollary}



\begin{corollary}
       \label{Corollary 02}
       Assume that $ f\in {\mathcal S}_{\bf M}$. Then for any $\tau >0$, $n\in {\mathbb N}$ and $\alpha>0$ the following inequality holds:
       \[
       E_n (f)\le 2C_{n,\alpha}(\tau )  \omega_\alpha\Big(f, \frac {\tau }n\Big),
       \]
       where the quantity $C_{n,\alpha}(\tau )$  is defined by (\ref{C_n,varphi,p}) with $\varphi(t)=2^\alpha
       |\sin (t/2) |^\alpha$.
\end{corollary}


For moduli of smoothness $\omega_\alpha(f,\delta)_{_{\scriptstyle  {\bf M}}}$, in the mentioned above spaces  ${\mathcal S}_{M}$  and ${\mathcal S}_{_{\scriptstyle  \mathbf p,\,\mu}}$, the inequalities
of the type (\ref{En<omega_M}) were proved  in \cite{Chaichenko_Shidlich_Abdullayev_2019} and
\cite{Abdullayev_Chaichenko_Imash kyzy_Shidlich_2020} correspondingly. 
Unlike to \cite{Chaichenko_Shidlich_Abdullayev_2019} and \cite{Abdullayev_Chaichenko_Imash kyzy_Shidlich_2020}, here we find the constant $C_{n,\varphi}(\tau )$ in Jackson-type inequality (\ref{En<omega}). Let us see how accurate this constant is.
For this, consider the case where all functions $M_k(u)=u^p\Big(p^{-1/p}q^{-1/q}\Big)^p$, $p>1$, $1/p+1/q=1$.
  In this case, all functions $\tilde{M}_k(v)=v^q$, the set $\Lambda$ is a set of all sequences of
  positive numbers $\lambda=\{\lambda_k\}_{k\in \mathbb{Z}}$ such that  $\|\lambda\|_{l_q({\mathbb Z})}\le 1$.
  Then the spaces ${\mathcal S}_{\bf M}$ coincide with the spaces ${\mathcal S}^{p}$, $p>1$,
  and by H\"{o}lder inequality  for any $f\in {\mathcal S}^{p}$,  the following relation holds:
  \[
    \|f\|^\ast_{_{\scriptstyle  {\bf M}}}= \sup\limits_{\lambda\in \Lambda} \sum\limits_{k \in \mathbb{Z}}
    \lambda_k|\widehat{f}(k) | \le \sup\limits_{\lambda\in \Lambda}
    \|\lambda \|_{l_p({\mathbb Z})} \cdot \|\{\widehat f({k})\}_{{k}\in\mathbb  Z}
    \|_{l_p({\mathbb Z})} \le \|f\|_{_{\scriptstyle {p}}}.
  \]
Furthermore, if   $f\!\not\equiv 0$, then for the sequence $\lambda^*_k=|\widehat{f}(k)|^{p/q}\Big(\sum_{j\in {\mathbb Z}}|\widehat{f}(k)|^{p}\Big)^{-1/q}$, $k\in {\mathbb Z}$, we have
 $
 \sum_{k \in \mathbb{Z}}     \lambda^*_k|\widehat{f}(k) |=\|f\|_{_{\scriptstyle {p}}}$ 
  and
 $\|\lambda^*\|_{l_q({\mathbb Z})}= 1.
 $ 
Therefore,  in this case $\|f\|^\ast_{_{\scriptstyle  {\bf M}}}=
\|f\|_{_{\scriptstyle {p}}}
$, $p>1$.

In the case $p=1$, the similar equality for norms
 \begin{equation} \label{Norm_eq}
 \| f\|^\ast_{_{\scriptstyle  {\bf M}}}=
\|f\|_{_{\scriptstyle {1}}}
\end{equation}
 obviously can be obtained if we consider all $M_k(u)=u$, $k\in {\mathbb Z}$, and the set
$\Lambda$ is a set of all sequences of  positive numbers $\lambda=\{\lambda_k\}_{k\in \mathbb{Z}}$ such that  $\|\lambda\|_{l_\infty({\mathbb Z})}=\sup_{k\in \mathbb{Z}}\lambda_k \le 1$.

For  fixed $n\in {\mathbb N}$, $\tau>0$  and  for a given $\varphi\in \Phi$, consider the quantity
 \[
    K_{n,\varphi}(\tau)_{_{\scriptstyle {p}}}
    :=\mathop{\sup\limits_{f\in {\mathcal S}^p}} \limits_{f\not\equiv {\rm const}}
     \frac{E_n (f)_{_{\scriptstyle {p}}}}
    {\omega_\varphi(f,\tau/n)_{_{\scriptstyle {p}}}} =
    \mathop{\sup\limits_{f\in {\mathcal S}^p}} \limits_{f\not\equiv {\rm const}}
     \frac{\inf\limits_{{t}_{n-1}\in {\mathcal T}_{n-1} }\|f-{t}_{n-1}\|_{_{\scriptstyle {p}}}}
    {\sup\limits_{|h|\le \delta} \|\Delta_h^\varphi f\|_{_{\scriptstyle {p}}}} .
 \]



\begin{theorem}
       \label{S^p_Theorem}
       Assume that $ f\in {\mathcal S}^p$, $1\le p<\infty$.
       Then for any $\tau >0$, $n\in {\mathbb N}$ and $\varphi\in \Phi$
              , the following inequality holds:
      \begin{equation}\label{En<omega_Sp}
      E_n (f)_{_{\scriptstyle {p}}} \le C_{n,\varphi,p}(\tau )\,
      \omega_\varphi\Big(f,\frac \tau n\Big)_{_{\scriptstyle {p}}},
      \end{equation}
      where
      \begin{equation}\label{C_n,varphi,p}
       C_{n,\varphi,p}(\tau ):= \bigg(\inf\limits _{v  \in  V(\tau )} \frac {v   (\tau ) - v   (0)}{
        I_{n,\varphi,p}(\tau ,v   )}\bigg)^{1/p} ,
      \end{equation}
      and
      \begin{equation}\label{I_n,varphi,p}
      I_{n,\varphi,p}(\tau ,v   ):=
      \inf\limits _{k \in {\mathbb N}:k \ge n} \int\limits _0^{\tau }\varphi^p\Big(\frac {k u}n\Big) dv   (u).
      \end{equation}
      In this case, there exists a function $v^*  \in  V(\tau )$ that realizes the greatest lower bound in (\ref{I_n,varphi,p}).
      Inequality  (\ref{En<omega_Sp}) is unimprovable on the set of all functions
      $f\in {\mathcal S}^{p}$, $f\not \equiv {\rm const}$,   in the sense that
      the following equality is true:
      \begin{equation}\label{S^p_equality}
       K_{n,\varphi}(\tau)_{_{\scriptstyle {p}}}=C_{n,\varphi,p}(\tau ).
       \end{equation}
\end{theorem}


 \begin{proof}   Here, we basically use
 the arguments given in \cite{Babenko_1986}, \cite{Chernykh_1967}, \cite{Chernykh_1967_MZ} and
  \cite{Stepanets_Serdyuk_2002}.   Let $ f\in {\mathcal S}^p$, $n\in {\mathbb N}$ and $h\in {\mathbb R}$.   By virtue of
   (\ref{modulus_generalize difference_Fourier_Coeff}) and (\ref{norm_Sp}), we have
 \[
 \|\Delta _h^\varphi f\|_{_{\scriptstyle {p}}}^p \ge \sum\limits_{|k|\ge n}
    \varphi^p(kh)|\widehat{f}(k) |^p
   \]
\[
   = \frac{I_{n,\varphi,p}(\tau ,v   )}{v (\tau)-v (0)}\sum\limits_{|k|\ge n}
      |\widehat{f}(k) |^p+   \sum\limits_{|k|\ge n}
      |\widehat{f}(k) |^p \Big(\varphi^p(kh)-\frac{I_{n,\varphi,p}(\tau ,v   )}{v (\tau)-v (0)}\Big).
\]
For  any $ u\in [0,\tau]$, we get
 \begin{equation}\label{diff_est_Sp}
    \|\Delta _{\frac un}^\varphi f\|_{_{\scriptstyle {p}}}^p \ge
     \frac{I_{n,\varphi,p}(\tau ,v   )}{v (\tau)-v (0)}\sum\limits_{|k|\ge n}
      |\widehat{f}(k) |^p
         +   \sum\limits_{|k|\ge n}  |\widehat{f}(k) |^p \bigg(\varphi^p\Big(\frac {k u}n\Big)-\frac{I_{n,\varphi,p}(\tau ,v   )}{v (\tau)-v (0)}\bigg).
 \end{equation}
The both sides of inequality (\ref{diff_est_Sp}) are nonnegative  and, in view of the
boundedness of the function $\varphi$, the series on its right-hand side is majorized on the entire
real axis by the absolutely convergent series
  $C^p(\varphi)\sum_{|k|\ge n}  |\widehat{f}(k) |^p$, where $C(\varphi):=\max_{u\in {\mathbb R}} \varphi(u)$. Then
 integrating this inequality with respect to $dv  (u)$  from  $0$  to $\tau,$ we get
\begin{equation}\label{diff_est_2_Sp}
    \int\limits  _0^{\tau }\|\Delta _{\frac un}^\varphi f\|_{_{\scriptstyle {p}}}^pdv  \ge
     I_{n,\varphi,p}(\tau ,v   )\sum\limits_{|k|\ge n}   |\widehat{f}(k) |^p
     $$
     $$
    +   \sum\limits_{|k|\ge n}
     |\widehat{f}(k) |^p \bigg(\int\limits  _0^{\tau } \varphi^p\Big(\frac {k u}n\Big)dv - I_{n,\varphi,p}(\tau ,v   ) \bigg).
 \end{equation}
By virtue of the definition of $I_{n,\varphi,p}(\tau ,v   )$, we see that the second term on the right-hand side of (\ref{diff_est_2_Sp}) is nonnegative. Therefore, for any function $v  \in  V(\tau )$,  we have
\[
    \int\limits  _0^{\tau }\|\Delta _{\frac un}^\varphi f\|_{_{\scriptstyle {p}}}^p dv  \ge
     I_{n,\varphi,p}(\tau ,v   )\sum\limits_{|k|\ge n}   |\widehat{f}(k) |^p \ge
     I_{n,\varphi,p}(\tau ,v   ) E_n^p (f)_{_{\scriptstyle {p}}}.
     \]
  Hence,
\begin{equation}\label{Module_Est_NEWWWW}
 E_n^p (f)_{_{\scriptstyle {p}}} \le \frac 1{I_{n,\varphi,p}(\tau ,v   )}\int\limits  _0^{\tau }\|\Delta _{\frac un}^\varphi f\|_{_{\scriptstyle {p}}}^pdv
\le \frac 1{I_{n,\varphi,p}(\tau ,v   )}\int\limits  _0^{\tau }\omega_\varphi^p \Big(f,\frac un\Big)_{_{\scriptstyle {p}}} dv.
 \end{equation}
whence  taking into account nondecreasing of the function $\omega_\varphi$,
 we immediately obtain relation (\ref{En<omega_Sp}) and  the estimate
   \begin{equation}\label{K_n<C_n}
   K_{n,\varphi}(\tau)_{_{\scriptstyle  {p}}} \le C_{n,\varphi,p}(\tau ).
 \end{equation}



Let us show that relation  (\ref{K_n<C_n}) is the equality.  By virtue of Lemma \ref{Lemma_Best_app}, we have
 \begin{equation}\label{K_n_new}
   K_{n,\varphi}(\tau)_{_{\scriptstyle  {p}}}
   =\mathop{\sup\limits_{f\in {\mathcal S}^p}}\limits_{f\not\equiv {\rm const}}
   \frac { \sum _{|k| \ge n}  |\widehat{f}(k)|^p}
    {\sup _{|h|\le \tau}\sum _{|k| \ge n}  \varphi^p(kh/n) |\widehat{f}(k)|^p}.
 \end{equation}
and in (\ref{K_n_new}), it is sufficient to consider supremum over all functions $f\in {\mathcal S}^p$ such that $\sum _{|k| \ge n}  |\widehat{f}(k)|^p\le 1$. Therefore, taking into account the parity of the function $\varphi$, we get
 \begin{equation}\label{K_n_new1}
   K_{n,\varphi}^{-p}(\tau)_{_{\scriptstyle  {p}}} \le
   J_{n,\varphi,p}(\tau):=\inf\limits_{w\in W_{n,\varphi,p}}\|w\|_{_{\scriptstyle  C_{[0,\tau]}}},
 \end{equation}
where the set
\begin{equation}\label{W_{n,varphi}}
 W_{n,\varphi,p}:=\Big\{\omega(u)=\sum_{j=n}^\infty \varrho_j \varphi^p(ju/n): \varrho_j\ge 0, \ \sum_{j=n}^\infty \varrho_j=1 \Big\}.
 \end{equation}
 For what follows, we need a duality relation in the space $C_{[a,b]},$  (see, e.g., \cite[Ch.~1.4]{Korneichuk 1987}).

\proclaim{Proposition A}\cite[Ch.~1.4]{Korneichuk 1987}
         \label{Duality_relation}
        If $F$ is a convex set in the space  $C_{[a,b]},$ then for any function $x\in C_{[a,b]}$,
        \begin{equation}\label{(6.38)}
        \inf\limits_{u\in F}\|x-u\|_{_{C_{[a,b]}}}=\sup\limits _{{{\mathop {V}\limits_a^b}}(g)\le 1}\Big( \int\limits _a^bx(t)dg(t)-\sup\limits  _{u\in F}\int\limits _a^bu(t)dg(t)\Big).
        \end{equation}
        For $x\in C_{[a,b]}\setminus \bar F$, where $\bar F$ is the closure of a set $F$,  there exists a function $g_*$ with variation equal to 1 on $[a,b]$ that realizes the least upper bound in (\ref{(6.38)}).
\endproclaim

It is easy to show that the set $W_{n,\varphi,p}$  is a convex   subset of the space
$C_{[0, \tau]}$. Therefore, setting  $a=0,$ $b=\tau,$ $x(t)\equiv
0,$ $u(t)=w(t) \in W_{n,\varphi,p},$ $F=W_{n,\varphi,p},$ from relation (\ref{(6.38)}) we get
 \begin{equation}\label{(6.39)}
   J_{n,\varphi,p}(\tau )=\inf\limits _{w\in W_{n,\varphi,p}}\|0-w\|_{_{C_{[0, \tau ]}}}
   $$
   $$
   =\sup\limits _{\mathop {V}\limits _0^{\tau}(g)\le 1}\Big(0
   -\sup\limits _{w\in W_{n,\varphi,p}}\int\limits _0^{\tau}w(t)dg(t)\Big)= \sup\limits _{\mathop {V}\limits _0^{\tau }(g)
   \le 1}\inf\limits _{w\in W_{n,\varphi,p}}\int\limits _0^{\tau } w(t)dg(t).
 \end{equation}
Furthermore, according to the Proposition A, there exists a function $g_*(t),$ that realizes the least upper bound in (\ref{(6.39)}) and such that $\mathop {V}\limits _0^{\tau }(g_*)=1$. Since every function $w\in W_{n,\varphi,p}$ is nonnegative, it suffices to take the supremum on the right-hand side of  (\ref{(6.39)})
over the set of nondecreasing functions $v  (t)$ for which $v  (\tau ) - v  (0)\le 1.$ For such functions, by virtue of (\ref{I_n,varphi}) and (\ref{W_{n,varphi}}), the following equality is true:
 \begin{equation}\label{(6.40)}
 \inf\limits _{w\in W_{n,\varphi,p}}\int\limits _0^{\tau }w(t)dv   (t)=I_{n,\varphi,p}(\tau ,v  ).
 \end{equation}
Hence, there exists a function $v _*\in V(\tau )$ such that $v _*(\tau )-v_*(0)=1$ and
  \begin{equation}\label{(6.41)}
  I_{n,\varphi,p}(\tau ,v _*)=\sup\limits _{v \in  V(\tau ):\mathop {V}\limits _0^{\tau }(v  )\le 1} I_{n,\varphi,p}(\tau ,v  )=J_{n,\varphi,p}(\tau ).
  \end{equation}
From relations    (\ref{K_n_new1}) and (\ref{(6.41)}), we get the necessary estimate:
  \[
     K_{n,\varphi}^p(\tau)_{_{\scriptstyle  {p}}}\ge \frac 1{
  J_{n,\varphi,p}(\tau )
  }
  =\frac 1{ I_{n,\varphi,p}(\tau ,v  _*)}=\frac {v _*(\tau )-v_*(0)} {I_{n,\varphi,p}(\tau ,v_*)}=C_{n,\varphi,p}^p(\tau ).
\]
$\hfill\Box$
 \end{proof}


From Theorem \ref{S^p_Theorem}, in particular, follows that  the  constant $C_{n,\varphi}(\tau )=
C_{n,\varphi,1}(\tau )$  is  exact in the Jackson-type inequality (\ref{En<omega}) in the case when $ {\mathcal S}_{\bf M}={\mathcal S}^1$. In this case, estimate (\ref{K_n<C_n}) in the proof  obviously  follows  directly from estimate (\ref{En<omega}) and relation (\ref{Norm_eq}). For $p>1$, estimate (\ref{K_n<C_n}) is more accurate than the estimate that can be obtained using similar arguments from Theorem \ref{Theorem_2.1}.


In the Lebesgue space $L_2({\mathbb T})$,   such result  was proved  for   ordinary moduli of smoothness $\omega_\alpha(f,\delta)_{_{\scriptstyle {p}}}$ with $\alpha= 1$ by  Babenko \cite{Babenko_1986}. In the spaces ${\mathcal S}^p$, for  moduli  $\omega_\alpha(f,\delta)_{_{\scriptstyle {p}}}$,  this theorem  was proved
by Stepanets and Serdyuk \cite{Stepanets_Serdyuk_2002}.
In the spaces ${\mathcal S}^p({\mathbb T}^d)$ of functions of several variables,  for  moduli  $\omega_\alpha(f,\delta)_{_{\scriptstyle {p}}}$,
such  result was obtained 
in \cite{Abdullayev_Ozkartepe_Savchuk_Shidlich_2019}. For generalized moduli of smoothness,  the similar result was proved  by  Vasil'ev \cite{Vasil'ev 2001} in $L_2({\mathbb T})$. We also mention the paper of Vakarchuk \cite{Vakarchuk 2016} which, in particular, contains a survey of the main results on Jackson-Type inequalities with generalized moduli of smoothness in the spaces $L_2({\mathbb T})$.




\section{ Inverse approximation theorem.}

 \begin{theorem}
       \label{Theorem_2}
       Let   $ f\in {\mathcal S}_{\bf M}$, the function $\varphi\in \Phi$ is  nondecreasing on an interval $[0,\tau]$ and
       $\varphi(\tau)=\max\{\varphi(t):t\in {\mathbb R}\}$. Then for any $n\in {\mathbb N}$,
       the following inequality holds:
       \begin{equation}\label{S_M.12}
       \omega _\varphi\Big(f, \frac{\tau}{n}\Big)
       \le    \sum _{\nu =1}^{n}\Big(\varphi\Big(\frac {\tau \nu}n\Big)-\varphi\Big(\frac {\tau (\nu-1)}n\Big)\Big)
       E_{\nu} (f)
       .
       \end{equation}
 \end{theorem}

\begin{proof}
 Let us use the proof scheme from \cite{Stepanets_Serdyuk_2002}, modifying it taking into account the peculiarities of the spaces
 ${\mathcal S}_{\bf M}$ and the definition of the modulus of smoothness $\omega _\varphi$.

 Let $ f\in {\mathcal S}_{\bf M}$. For any $\varepsilon>0$ there exist a number $N_0=N_0(\varepsilon)\in {\mathbb N}$, $N_0>n$,
such that for any $N>  N_0$, we have
  \[
  E_N(f)
  =\|f-{S}_{N-1}({f})\|
  <\varepsilon/\varphi(\tau).
  \]
 Let us set $f_{0}:=S_{N_0}(f)$. Then in view of  (\ref{modulus_generalize difference_Fourier_Coeff}), we see that
\begin{equation}\label{(6.3100)}
\|\Delta_h^\varphi f\|
\le    \|\Delta_h^\varphi f_{0}\|
+\|\Delta_h^\varphi (f-f_{0})\|
\le  \|\Delta_h^\varphi f_0\|
 +
 \varphi(\tau) E_{N_0+1}(f)
 <\|\Delta_h^\varphi f_0\|
 +\varepsilon.
\end{equation}
Further, let ${S}_{n-1}:={S}_{n-1}(f_0)$  be the Fourier sum of $f_0$.  Then  by virtue of (\ref{modulus_difference_Fourier_Coeff}), for $|h|\le \tau /n$, we have
\[
 \|\Delta_h^\varphi f_0\|
 =\|\Delta_h^\varphi (f_0-S_{n-1})+\Delta_h^\varphi S_{n-1}\|
 \le\Big\| \varphi(\tau)(f_0-S_{n-1})+\sum _{|k|\le n-1}\varphi(kh)|\widehat{f}(k)|
 {\mathrm{e}^{\mathrm{i}k\cdot}}\Big\|
\]
\begin{equation}\label{(6.73)}
\le  \Big\|\varphi(\tau) \sum _{\nu =n}^{N_0} H_{\nu} +
     \sum _{\nu =1}^{n-1}\varphi\Big(\frac {\tau \nu}n\Big)H_{\nu}  \Big\|
     ,
\end{equation}
where $H_{\nu}(x) :=H_{\nu}(f,x)=|\widehat{f}(\nu)| {\mathrm{e}^{\mathrm{i}\nu x}}+
|\widehat{f}(-\nu)| {\mathrm{e}^{-\mathrm{i}\nu x}}$, $\nu=1,2,\ldots$

Now we use the following assertion which is proved directly.

 \begin{lemma}
       \label{Lemma_31}
       Let  $\{c_{\nu}\}_{\nu=1}^\infty$ and $\{a_{\nu}\}_{\nu=1}^\infty$ be arbitrary numerical sequences.
       Then the following equality holds for  all natural $m$, $M$ and $N$ $m\le M<N$:
       \begin{equation}\label{(6.74)}
       \sum _{\nu =m}^Ma_{\nu }c_{\nu }=a_m\sum _{\nu=m}^{N }c_{\nu }+
       \sum _{\nu =m+1}^M(a_{\nu } -a _{\nu-1})\sum _{i=\nu }^{N }c_i-a_M\sum _{\nu =M+1}^{N }c_{\nu}.
       \end{equation}
 \end{lemma}

Setting  $a_{\nu }=\varphi\Big(\frac {\tau \nu}n\Big),$  $c_{\nu }=H_{\nu}(x), $ $m=1$,  $M=n-1$ and $N=N_0$ in (\ref{(6.74)}), we get
 \[
  \sum _{\nu =1}^{n-1}\varphi\Big(\frac {\tau \nu}n\Big)H_{\nu}(x)=
   \sum _{\nu =1}^{N_0}H_{\nu}(x)
   \]
   \[
   +
   \sum _{\nu =2}^{n-1}\bigg(\varphi\Big(\frac {\tau \nu}n\Big)-\varphi\Big(\frac {\tau (\nu-1)}n\Big)\bigg)
    \sum_{i=\nu }^{N_0}H_{i}(x) -\varphi\Big(\frac {\tau (n-1)}n\Big)\sum
_{\nu =n}^{N_0 }H_{\nu}(x).
  \]
Therefore,
  \[
 \bigg\|\varphi(\tau) \sum _{\nu =n}^{N_0} H_{\nu} +
     \sum _{\nu =1}^{n-1}\varphi\Big(\frac {\tau \nu}n\Big)H_{\nu} \bigg\|
  \]
  \[
   \le
        \bigg\|\varphi(\tau) \sum _{\nu =n}^{N_0} H_{\nu} +\sum _{\nu =1}^{n-1}\bigg(\varphi\Big(\frac {\tau \nu}n\Big)-\varphi\Big(\frac {\tau (\nu-1)}n\Big)\bigg)
    \sum_{i=\nu }^{N_0}H_{i}  -\varphi\Big(\frac {\tau (n-1)}n\Big)\sum
_{\nu =n}^{N_0}H_{\nu}  \bigg\|
\]
\begin{equation}\label{(6.74aqq)}
 \le \bigg\|\sum _{\nu =1}^{n}\bigg(\varphi\Big(\frac {\tau \nu}n\Big)-\varphi\Big(\frac {\tau (\nu-1)}n\Big)\bigg)
    \sum_{i=\nu }^{N_0}H_{i}
    \bigg\|
    $$
    $$
    \le    \sum _{\nu =1}^{n}
    \bigg(\varphi\Big(\frac {\tau \nu}n\Big)-\varphi\Big(\frac {\tau (\nu-1)}n\Big)\bigg) E_{\nu} (f_0)
    .
  \end{equation}
Combining relations (\ref{(6.3100)}), (\ref{(6.73)})  and (\ref{(6.74aqq)}) and taking into account the definition of the function $f_0$, we see that for  $|h|\le \tau /n$, the following inequality holds:
\[
  \|\Delta_h^\varphi f\|
  \le
    \sum _{\nu =1}^{n}\Big(\varphi\Big(\frac {\tau \nu}n\Big)-\varphi\Big(\frac {\tau (\nu-1)}n\Big)\Big)
    E_{\nu} (f)
    +\varepsilon
\]
 which, in view of arbitrariness of $\varepsilon$,  gives us (\ref{S_M.12}). $\hfill\Box$
\end{proof}

As noted  above, for  $\varphi(t)=2^\alpha |\sin (t/2) |^\alpha$, $\alpha>0$, we have
$\omega_\varphi(f,\delta)=\omega_\alpha(f,\delta)$. In this case, the number $\tau=\pi$.
If $\alpha\ge 1$, then using the inequality $x^\alpha-y^\alpha \le \alpha x^{\alpha-1}(x-y),$ $x>0, y>0$
(see, for example, \cite[Ch.~1]{Hardy_Littlewood_Polya_1934}), 
and the usual trigonometric formulas, for $\nu=1,2,\ldots,n,$ we get
\[
  \varphi\Big(\frac {\tau \nu}n\Big)-
  \varphi\Big(\frac {\tau (\nu-1)}n\Big)=
  2^\alpha \Big(\Big|\sin \Big(\frac {\pi \nu}n\Big) \Big|^\alpha-
  \Big|\sin \Big(\frac {\pi (\nu-1)}n\Big) \Big|^\alpha\Big)\le
\]
\[
    \le 2^\alpha \alpha |\sin \Big(\frac {\pi \nu}n\Big) \Big|^{\alpha-1}
    \Big|\sin \Big(\frac {\pi \nu}n\Big) -
  \sin \Big(\frac {\pi (\nu-1)}{n}\Big) \Big|
  \le \alpha \Big(\frac{2\pi }{n}\Big)^\alpha \nu^{\alpha-1}.
\]
If  $0<\alpha<1$, then the similar estimate can be obtained  using the inequality $x^\alpha-y^\alpha \le \alpha y^{\alpha-1}(x-y)$,
which holds for any $x>0, y>0,$ \cite[Ch.~1]{Hardy_Littlewood_Polya_1934}. 
 Hence, we get the following statement:

 \begin{corollary}
       \label{Corollary 21}
       Let   $ f\in {\mathcal S}_{\bf M}$ and $\alpha>0$. Then for any $n\in {\mathbb N}$,
       the following inequality holds:
       \begin{equation}\label{Inverse_Inequality}
       \omega _\alpha \Big(f, \frac{\pi}{n}\Big)\le      \alpha \Big(\frac{2\pi }{n}\Big)^\alpha
       \sum _{\nu =1}^{n} \nu^{\alpha-1} E_{\nu} (f).
       \end{equation}
\end{corollary}

Note that in the above-mentioned spaces  ${\mathcal S}_{M}$ and ${\mathcal S}_{_{\scriptstyle  \mathbf p,\,\mu}}$,
the similar estimates were obtained for moduli of smoothness and  best approximations determined  with respect to the corresponding Luxemburg norms in \cite{Abdullayev_Chaichenko_Imash kyzy_Shidlich_2020}
and  \cite{Chaichenko_Shidlich_Abdullayev_2019}. In ${\mathcal S}^p$, such estimates were obtained in  \cite{Sterlin_1972} and \cite{Stepanets_Serdyuk_2002}.  For the Lebesgue spaces $L_p$, inequalities of the type (\ref{Inverse_Inequality}) were proved by M.~Timan (see, for example, \cite[Ch. 2]{M_Timan_M2009}, \cite[Ch. 6]{A_Timan_M1960}).


\begin{corollary}
      \label{Corollary 2}
      Assume that the sequence of the best approximations
      $E_n(f)
      $ of a function $f\in {\mathcal S}_{\bf M}$ satisfies the following relation for some $\beta >0$:
      \[
      E_n(f)
      = {\mathcal O}(n^{-\beta }).
      \]
      Then, for any $\alpha>0$, one has
      \[
      \omega _\alpha(f, t)
      =    \left \{ \begin{matrix}  {\mathcal O}(t^{\beta }) & \hfill \mbox {for} \ \ \beta <\alpha, \hfill \cr
                        {\mathcal O}(t^\alpha|\ln t|) & \hfill \mbox {for}\ \ \beta =\alpha, \hfill \cr
                        {\mathcal O}(t^\alpha) & \hfill \mbox {for} \ \ \beta >\alpha.\hfill
                        \end{matrix} \right.
     \]
\end{corollary}



\section{ Constructive characteristics of the classes of functions defined by the
$\alpha$th moduli of smoothness}

In this section we give the constructive characteristics of the classes
${\mathcal S}_{\bf M}H^{\omega}_{\alpha} $ of functions for which the $\alpha$th moduli of smoothness
$\omega_\alpha(f,\delta)$ do not exceed some majorant.

Let $\omega$ be  a function defined on interval $[0,1]$. For a fixed $\alpha>0$, we set
\begin{equation} \label{omega-class}
    {\mathcal S}_{\bf M}H^{\omega}_{\alpha}=
    \Big\{f\in {\mathcal S}_{\bf M}:  \quad \omega_\alpha(f, \delta) =
    {\mathcal O}  (\omega(\delta)),\quad  \delta\to 0+\Big\}.
\end{equation}
Further, we consider the functions $\omega(\delta)$, $\delta\in [0,1]$, satisfying the following conditions 1)--4): \textbf{1)} $\omega(\delta)$ is continuous on $[0,1]$;\  \textbf{  2)} $\omega(\delta)\uparrow$;\  \textbf{  3)} $\omega(\delta)\not=0$ for any $\delta\in (0,1]$;\  \textbf{  4)}~$\omega(\delta)\to 0$ as $\delta\to 0+$; and the well-known  condition $({\mathcal B}_\alpha)$, $\alpha>0$ (see, e.g. \cite{Bari_Stechkin_1956}):
 \[
 ({\mathcal B}_\alpha): \quad\quad  {\sum_{v=1}^n v^{\alpha-1}\omega({v^{-1}}) =
{\mathcal O}  (n^\alpha \omega ( {n^{-1}}))},\quad n\to \infty.
 \]

\begin{theorem}
       \label{Theorem_3}
       Assume that $\alpha>0$  and the function $\omega$ satisfies  conditions  $1)$--\,$4)$ and $({\mathcal  B}_\alpha)$. Then, in order a function $f\in {\mathcal S}_{\bf M}$ to belong to the class ${\mathcal S}_{\bf M}H^{\omega}_{\alpha}$, it is necessary and sufficient that
       \begin{equation} \label{iff-theorem}
       E_n(f) ={\mathcal O} ( \omega ({n^{-1}} ) ).
       \end{equation}
\end{theorem}

\begin{proof}
 Let $f \in {\mathcal S}_{\bf M}H^{\omega}_{\alpha}$, by virtue of Corollary \ref{Corollary 02}, we have
\begin{equation} \label{using-direct-theorem}
    E_n(f) \le 2C_{n,\alpha}(1) \omega_\alpha (f; {n^{-1}} ),
\end{equation}
Therefore, {relation (\ref{omega-class}) yields (\ref{iff-theorem})}.
On the other hand, if relation (\ref{iff-theorem}) holds, then by virtue of (\ref{Inverse_Inequality}),
taking into account the condition $({\mathcal  B}_\alpha)$, we obtain
\begin{equation} \label{using-inverse-theorem}
    \omega _\alpha(f, {n^{-1}} )\le
    \alpha \Big(\frac{2\pi }{n}\Big)^\alpha \sum _{\nu =1}^n \nu ^{\alpha-1} E_{\nu}(f)
        \le
    \frac {C }{n^\alpha} \sum _{\nu =1}^n \nu ^{\alpha-1} \omega ({v^{-1}})=
    {\mathcal O}  (\omega ( {n^{-1}})).
\end{equation}
Thus, the function  $f$ belongs to the set ${\mathcal S}_{\bf M}H^{\omega}_{\alpha}$.$\hfill\Box$

\end{proof}

The function $h(t)=t^r$, $r \le \alpha$, satisfies the condition $({\mathcal  B}_\alpha)$. Hence, denoting by ${\mathcal S}_{\bf M}H_{\alpha}^r$  the class ${\mathcal S}_{\bf M}H^{\omega}_{\alpha}$ for $\omega(t)=t^r$, $0<r\le \alpha,$  we establish the following statement:

\vskip 2mm

\begin{corollary}
      \label{Corollary_5}
      Let $\alpha >0$, $0<r\le \alpha.$
      In order  a function $f \in S{_{\scriptstyle {\bf M}}}$ to belong to ${\mathcal S}_{\bf M}H_{\alpha}^r$, it is
      necessary and sufficient that
      \[
      E_n(f)
      ={\mathcal O}   ({n^{-r}} ).
      \]
\end{corollary}





\end{document}